\documentclass{article}

\usepackage{amsmath}
\usepackage{amssymb}
\usepackage{amsthm}
\usepackage{mathrsfs}
\usepackage{wasysym}
\usepackage{enumitem}
\usepackage{mathtools}
\usepackage{stmaryrd}
\usepackage[nottoc]{tocbibind}	
\usepackage[dvipsnames]{xcolor}
\usepackage[a4paper, left=2.7cm, right=2.7cm, top=3.5cm, bottom=2cm]{geometry}
\usepackage[pdftex=true,hyperindex=true,pdfborder={0 0 0}]{hyperref}
\usepackage{doi}
\usepackage[numbers]{natbib}
\usepackage{tikz}
\usepackage{array}
\usetikzlibrary{shapes,arrows,automata,matrix,fit}

\theoremstyle{plain}
\newtheorem{theorem}{Theorem}[section]
\newtheorem{lemma}[theorem]{Lemma}
\newtheorem{corollary}[theorem]{Corollary}
\newtheorem{proposition}[theorem]{Proposition}

\theoremstyle{definition}
\newtheorem{remark}[theorem]{Remark}
\newtheorem{example}[theorem]{Example}
\newtheorem{definition}{Definition}[section]
\newtheorem{open}[theorem]{Open question}

\newcommand{\N}{\mathbb N}	
\newcommand{\Z}{\mathbb Z}	
\newcommand{\Ns}{{\mathbb N}\setminus\{0\}}	

\renewcommand{\S}{\Sigma}
\newcommand{\Sk}{\Sigma_k}
\newcommand{\D}{\mathcal D}	
\newcommand{\Dr}{\Delta_r}	
\newcommand{\dr}{\delta^r}
\newcommand{\F}{{\mathcal F_r}}
\newcommand{\lb}{\llbracket}
\newcommand{\rb}{\rrbracket}
\DeclareMathOperator*{\card}{card}
\newcommand{\myvec}[1]{\ensuremath{\begin{pmatrix}#1\end{pmatrix}}}

\usepackage{comment}
\usepackage{xcolor}
\definecolor{darkgreen}{RGB}{20,200,10}

\begin{document}

\title{Discrete correlations of order 2 of generalised \\ Rudin--Shapiro sequences: 
a combinatorial approach\footnotetext{Last update:~\today}}

\author{Ir\`ene Marcovici, Thomas Stoll, Pierre-Adrien Tahay}

\date{}

\maketitle

\begin{abstract}
We introduce a family of block-additive automatic sequences, that are obtained by allocating a weight to each couple of digits, and defining the $n$th term of the sequence as being the total weight of the integer $n$ written in base $k$. Under an additional \emph{difference condition} on the weight function, these sequences can be interpreted as generalised Rudin--Shapiro sequences, and we prove that they have the same correlations of order $2$ as sequences of symbols chosen uniformly and independently at random. The speed of convergence is very fast and is independent of the prime factor decomposition of $k$. This extends recent work of Tahay~\cite{Ta20}. The proof relies on direct observations about base-$k$ representations of integers and combinatorial considerations. We also provide extensions of our results to higher-dimensional block-additive sequences. 

\medskip

\noindent
\emph{Keywords:} automatic sequences, pseudorandom sequences, Rudin--Shapiro sequences, difference matrices, discrete correlations

\smallskip
	
\renewcommand{\contentsname}{\vspace{1cm}}
{\footnotesize\tableofcontents}
\end{abstract}

\section{Introduction}

A $k$-automatic sequence on a finite set $G$ is a sequence $u\in G^{\N}$ that can be computed by a deterministic finite automaton with output (DFAO) in the following way: the $n$-th term of the sequence is a function of the state reached by the automaton after reading the representation of the integer $n$ in base $k$. Alternatively, a $k$-automatic sequence can also be defined as a sequence generated by a $k$-uniform morphism. We refer to the book of Allouche and Shallit~\cite{AlSh03} for a complete survey on automatic sequences.

Although automatic sequences are deterministic sequences having a very simple algorithmic description, some of them exhibit a complex behaviour. In this work, we are interested in exploring ``how random'' an automatic sequence can look like. There are many different ways to measure the ``random aspect'' of a deterministic sequence. Here, we will study families of automatic sequences having the same discrete correlations of order $2$ as sequences of symbols chosen uniformly and independently at random. We also provide explicit estimates for the speed of convergence. 

The sequences we will consider are \emph{block-additive sequences}. They are obtained by allocating a weight to each couple of digits, and defining the $n$the term of the sequence as being the total weight of the integer $n$ written in base $k$. This weight is obtained by sliding the representation of the integer $n$ in base $k$ with a window of length $2$ (or more generally, of length $\ell\geq 1$), and summing all the weights read. The name \emph{block-additive} was already used in previous articles~\cite{DrGrLi08,Mu18}. With the terminology of Cateland~\cite{Ca92}, these sequences are \emph{digital sequences}. In the special case where the weight matrix is a \emph{difference matrix}, we will say that the automatic sequence obtained is a \emph{generalised Rudin--Shapiro sequence}, and prove that it has the same correlations of order $2$ as a sequence of symbols chosen uniformly and independently at random.

As we will comment on further in the article, our terminology of \emph{generalised Rudin--Shapiro sequences} is consistent with the definitions of~\cite{GrShSt09, Ta20}, and also intersects previous notions of generalised Rudin--Shapiro sequences, such as the one of Qu\'effelec~\cite{Qu87} (see~\cite{GrShSt09} for further references). For other generalisations of the Rudin--Shapiro sequence that we will not investigate here, see Allouche and Shallit~\cite{AS93} and Mauduit and Rivat~\cite{MR15}.

As in the articles of Grant et al.~\cite{GrShSt09} and Tahay~\cite{Ta20}, we study the correlations of order $2$ of generalised Rudin--Shapiro sequences, but rather than making use of exponential sums, we here only employ direct arguments relying on the base-$k$ decomposition of the integers $n$ and $n+r$, for a fixed $r$. This approach highlights the combinatorial role played by the \emph{difference condition} defining a difference matrix, and allows to obtain more precise estimates on the correlations of order $2$. Furthermore, in addition to studying the asymptotic proportion of integers $n$ satisfying $u_n=u_{n+r}$, we provide results on the proportion of integers for which $(u_n,u_{n+r})=(i,j)$, for any possible value of the couple $(i,j)\in G^2$. Precisely, we prove that the limit is equal to $1/ |G|^2$ for all $(i,j)\in G^2$ and for any $r\in\Ns$, as for an i.i.d. sequence of symbols uniformly drawn in $G$. After considering the one-dimensional case, we also mention extensions of our results to higher-dimensional block-additive sequences.

\section{Definitions and presentation of the results}

In all the article, we denote by $\N$ the set of non-negative integers.

\subsection{Block-additive sequences of rank 2}

For $k\in\Ns$, we define $\Sk=\{0,\ldots,k-1\}$, and we denote by $[n]_k$ the representation of the integer $n\in\N$ in base $k$. By definition, it is the unique sequence $x=(x_i)_{i\in\N}\in \Sk^\N$ containing finitely many non-zero values, such that $$n=\sum_{i\in\N} x_i k^i.$$ 
We will write $$\begin{array}{rccccc} [n]_k=&x_0&x_1&x_2& x_3&\cdots\end{array}.$$
We also introduce the notation $\ell_n=\min\{i\in \N : \forall j>i, \, x_i=0\}$, and we define  $$\sigma_k(n)=\sum_{i\in\N} x_i=\sum_{i=0}^{\ell_n} x_i,$$
the $k$-ary sum-of-digits function.

\begin{definition}\label{def:blockadd}
Let $(G,+)$ be a finite abelian group, let $k\in\Ns$, and let $f:\Sk\times \Sk\to G$ be a function satisfying $f(0,0)=0$. We say that the sequence $u=(u_n)_{n\in\N}\in G^{\N}$ is a \emph{block-additive sequence (of rank 2) in base $k$} of \emph{weight function (or matrix)} $f$ if for any integer $n\in\N$, we have
$$u_n=\sum_{i\in\N} f(x_i,x_{i+1}),$$
where $[n]_k=x$.
\end{definition}

\begin{example}[Prouhet--Thue--Morse sequence] The Prouhet--Thue--Morse sequence is given by
$$\forall n\in\N, \quad u_n\equiv\sigma_2(n) \pmod{2}.$$
The Thue-Morse sequence is a block-additive sequence in base $k=2$, with $G=\Z_2$, and weight function $f:\S_2\times\S_2\to G$ defined by: $\forall (i,j)\in G^2, f(i,j)=i.$ \\
The first terms are given by $u=(0,1,1,0,1,0,0,1,1,0,0,1,0,1,1,0,\ldots)$.\\
We represent below a DFAO computing this sequence.
\begin{center}
\begin{tikzpicture}[->,>=stealth',shorten >=1pt,auto,node distance=2.8cm, semithick]
\node[initial,state] (A) {$q_0|0$};
\node[state] (B) [right of=A] {$q_1|1$};
\path (A) edge [loop above] node {0} (A)
(A) edge [bend left]  node {1} (B)
(B) edge [bend left] node {1} 
(A) edge [loop above] node {0} (B);
\end{tikzpicture}
\end{center}
\end{example}

\begin{example}[Classical Rudin--Shapiro sequence] The (classical) Rudin--Shapiro sequence on $G=\Z_2$ can be defined as the block-additive sequence in base $k=2$ of weight function $f:\S_2\times\S_2\to G$ given by $\forall (i,j)\in G^2, \quad f(i,j)=ij$.
In other words, $u_n$ gives the parity count of the number of (possibly overlapping) occurrences of the block $11$ in the binary expansion of $n$. \\
The first terms are given by $u=(0,0,0,1,0,0,1,0,0,0,0,1,1,1,0,1,\ldots)$. \\
\end{example}

The following proposition is straightforward, for the sake of completeness we include the proof.

\begin{proposition}\label{prop:aut} If a sequence is block-additive in base $k$, then it is a $k$-automatic sequence.
\end{proposition}

\begin{proof} Let $Q=G\times \Sk$, $q_0=(0,0)$, let $\delta:Q\times \S_k\to Q$ be defined by
$$\delta((g,i),j)=(g+f(j,i),j),$$
and let $\tau:Q\to G$ be defined by $\tau(g,i)=g$. 
The DFAO $(Q,\S_k,\delta,q_0,\tau)$ computes the block-additive sequence $u=(u_n)_{n\in\N}$ of weight function $f$, by reading the representation of the integer $n$ in base $k$ starting with the most significant digit, and using the output map $\tau$.
\end{proof}

\begin{remark}\label{rem:morphic} Alternatively, a block-additive sequence has the following morphic description. Let again $Q=G\times \Sk$ and $q_0=(0,0)$, and let $\phi:Q^*\to Q^*$ be the $k$-uniform morphism satisfying, for a state $s=(g,i)\in Q$, $\phi(s)=s_0\cdots s_{k-1}$, with $s_j=(g+f(j,i),j).$
Consider the fixed point $\phi^{\omega}(q_0)\in Q^{\N}$. Then, the letter-to-letter projection of $\phi^{\omega}(q_0)$ by $\tau$ is the block-additive sequence of the function $f$.
\end{remark}

\begin{example} We represent below the DFAO given by the proof of Prop.~\ref{prop:aut} for the (classical) Rudin--Shapiro sequence.
\begin{center}
\begin{tikzpicture}[->,>=stealth',shorten >=1pt,auto,node distance=2.8cm, semithick]
\node[initial,state] (A) {$(0,0)|0$};
\node[state] (B) [right of=A] {$(0,1)|0$};
\node[state] (C) [below of=B] {$(1,1)|1$};
\node[state] (D) [below of=A] {$(1,0)|1$};
\path (A) edge [loop above] node {0} (A)
(D) edge [loop left] node {0} (D)
(A) edge [bend left]  node {1} (B)
(B) edge [bend left]  node {1} (C)
(C) edge [bend left]  node {0} (D)
(D) edge [bend left]  node {1} (C)
(C) edge [bend left]  node {1} (B)
(B) edge [bend left]  node {0} (A) ;
\end{tikzpicture}
\end{center}
With the notations $q_0=(0,0), q_1=(0,1), q_2=(1,0), q_3=(1,1)$, the $2$-uniform morphism described above is here given by
$$\phi(q_0)=q_0q_1, \; \phi(q_1)=q_0q_2, \; \phi(q_2)=q_3q_1, \; \phi(q_3)=q_3q_2,$$
with $\tau(q_0)=\tau(q_1)=0, \tau(q_2)=\tau(q_3)=1.$
\end{example}

\subsection{Difference matrices and generalised Rudin--Shapiro sequences}

\begin{definition} Let $(G,+)$ be a finite abelian group, and let $k\in\Ns$. A \emph{difference matrix} of size $k$ is a matrix $D=(d(i,j))_{(i,j)\in \Sk\times\Sk}\in G^{\Sk\times\Sk}$ satisfying the following \emph{difference condition}
$$\forall (i,j)\in \Sk\times\Sk \mbox{ with }\; i\not= j, \quad \forall g\in G, \quad \card \Big\{h\in \Sk : d(i,h)-d(j,h) = g \Big\} ={k\over |G|}.$$
\end{definition}

In other words, $D$ is a difference matrix if for any $(i,j)\in\Sk\times\Sk$ with $i\neq j$, the set $\{d(i,h)-d(j,h) : h\in\Sk\}$ contains every element of $G$ equally often. Note that the difference condition requires the integer $k$ to be a multiple of $|G|$. We introduce the notation $\pi=k/|G|$, and we have thus $\pi\in\Ns$. We denote by $\D(G,k)$ the set of difference matrices of size $k$ over the group $G$.

\begin{definition}\label{def:grs} A block-additive sequence is a \emph{generalised Rudin--Shapiro sequence} if its weight function $f$ is such that the matrix $(f(i,j))_{(i,j)\in \Sk\times\Sk}\in G^{\Sk\times\Sk}$ is a difference matrix. 
\end{definition}

\begin{example} 
\begin{enumerate}
\item The Thue-Morse sequence is \emph{not} a generalised Rudin--Shapiro sequence, since its weight function is given by the matrix $\begin{pmatrix}0&0\\ 1&1\end{pmatrix}$, which does not belong to $\D(\Z_2,2)$.
\item The classical Rudin--Shapiro sequence is a generalised Rudin--Shapiro sequence, since its weight function is given by the matrix $\begin{pmatrix}0&0\\ 0&1\end{pmatrix}$, which belongs to $\D(\Z_2,2)$.
\end{enumerate}
\end{example}

Let us present different ways to construct difference matrices, and thus to define generalised Rudin--Shapiro sequences. 

\begin{example} Let $p$ be a prime number, and let $G=\Z_p$. Then, the matrix $D=(d(i,j))_{(i,j)\in \S_p\times\S_p}$ defined by $d(i,j)\equiv ij \pmod{p}$ is a difference matrix. The block-additive sequences thus obtained correspond to Queff\'elec's generalisation of the Rudin--Shapiro sequence~\cite[Section 4]{Qu87}.  By definition, if $[n]_p=x,$ we have $u_n\equiv\sum_{i\in\N}x_ix_{i+1}\pmod{p}$.
\begin{itemize}
\item As a particular case, for $p=2$, the difference matrix is given by $\begin{pmatrix}0&0\\ 0&1\end{pmatrix}$, and we recover the classical Rudin--Shapiro sequence.
\item For $p=3$, the difference matrix is given by $\begin{pmatrix}0&0&0\\ 0&1&2\\ 0&2&1 \end{pmatrix}$.
\end{itemize}
\end{example}

\begin{example} For $k=3$, another example of a difference matrix on $G=\Z_3$ is given by $\begin{pmatrix}0&1&1\\ 1&0&1\\ 1&1&0 \end{pmatrix}$. In the sequence obtained, the term $u_n$ counts (modulo 3) the number of blocks of distinct digits in the base-$3$ decomposition of the integer $n$.
\end{example}

It can be seen that for an even integer $k\geq 4$, there exists \emph{no} difference matrix of size $k$ on $G=\Z_k$. Indeed, if $k$ is even, we have $\sum_{i=0}^{k-1}i\equiv{k/2} \pmod{k}$. But if $\sum_{h=0}^{k-1} \left(d(i_1,h)-d(i_2,h)\right)\equiv{k/2} \pmod{k}$ and $\sum_{h=0}^{k-1} \left(d(i_2,h)-d(i_3,h)\right)\equiv {k/2} \pmod{k}$, then $\sum_{h=0}^{k-1} \left(d(i_1,h)-d(i_3,h)\right) \equiv 0 \pmod{k}$, so that we obtain a contradiction.
 
However, the following theorem shows the existence of difference matrices at least for all powers of prime numbers. We include the proof for the sake of clearness.

\begin{theorem}\label{thm:diffmat}\cite[Theorem 6.6]{HeSlSt99} For any prime number $p$ and any integers $m,n\in\Ns$ such that $m\leq n$, there exists a finite abelian group $G$ of order $p^m$ such that the set $\D(G,p^n)$ is non-empty.
\end{theorem}

\begin{proof} Let $H=\mathbb{F}_{p^m}$, and $G=\mathbb{F}_{p^n}$ be the finite fields with respectively $p^m$ and $p^n$ elements. We can represent the elements of $G$ by polynomials of the form $\beta_0+\beta_1x+\cdots+\beta_{n-1}x^{n-1},$ with $\beta_0,\ldots,\beta_{n-1} \in \mathbb{Z}_p$. The group $(H,+)$ can be seen as the subgroup of $(G,+)$ made of the polynomials of degrees smaller or equal to $m$. 
Let $\varphi : G \rightarrow H$ be the function which maps the element $\beta_0+\beta_1x+\cdots+\beta_{n-1}x^{n-1}$ to the element $\beta_0+\beta_1x+\cdots+\beta_{m-1}x^{m-1}$, and for two polynomials $(\alpha(x),\beta(x))\in G^2$, let $d(\alpha(x),\beta(x))=\varphi(\alpha(x)\cdot\beta(x))$, where $\cdot$ denotes the multiplication in the field $G$. Then, one can check that the matrix $D=(d(i,j))_{(i,j)\in \S_{p^n}\times\S_{p^n}}$ (we identify $\S_{p^n}$ with $G$, using any bijection) is a difference matrix on $(G,+)\cong(\Z_p^n,+)$.
\end{proof}

Note that there exist difference matrices which do not belong to the families described in the proof of Theorem~\ref{thm:diffmat} (see~\cite[p.127]{La15} and ~\cite[Table 6.37]{HeSlSt99}). For example, the matrix
$$\begin{pmatrix}
0&0&0&0&0&0\\
0&0&1&1&2&2\\
0&1&0&2&1&2\\
0&1&2&0&2&1\\
0&2&1&2&0&1\\
0&2&2&1&1&0
\end{pmatrix}$$
is an element of $\D(\Z_3,6)$ that is not covered by Theorem~\ref{thm:diffmat}. 

The enumeration and the classification of difference matrices is a complex task. We refer to~\cite{HeSlSt99,La15} for an indepth study of these questions and various examples of difference matrices.

\subsection{Main results}

We can now state our main results, in the one-dimensional case. We use the notation $\log_k(N)$ for the logarithm of $N$ to base $k$. 

\begin{theorem}\label{thm:main1} If $u$ is a generalised Rudin--Shapiro sequence, then for any $r\in\Ns$, $g\in G$, and $N\in\N$,
$${1\over N} \; \Big | \card\Big \{n \in \lb 0, N-1\rb : u_{n+r}-u_n=g \Big \} - {1\over |G|} \Big | \, \leq \, r\, k \, {1+\log_k(N)\over N}.$$
\end{theorem}

The limit $1/ |G|$ is thus the same as for an i.i.d. sequence of symbols uniformly distributed in $G$. But the convergence is here much faster than in the random case, since the error term is of order $\log(N)/N$, while for i.i.d. sequences, the central limit theorem tells us that it is in $1/\sqrt{N}$.

\begin{remark}
For $k$ prime or a prime power, the bound in Theorem~\ref{thm:main1} is the same as the one obtained by Tahay~\cite[Theorem 4]{Ta20}. This is natural since the underlying objects (generalisations of the Rudin--Shapiro sequence) are the same. However, our generalisation of the Rudin--Shapiro sequence to other composed $k$ is different from Tahay~\cite{Ta20}: it is directly based on one single difference matrix of size $k$, while Tahay's construction uses the prime factor decomposition of $k$ and, as a side effect, the error term in his result is $N^{-1/d}$ where $d$ denotes the number of different primes appearing in the prime factor decomposition of $k$~\cite[Theorem 5]{Ta20}. The size of our error term for our generalised objects is ${\log (N)}/N$, as $N\rightarrow \infty$, which is much smaller for fixed $r$ and is independent of the arithmetic structure of $k$.
\end{remark}

\begin{theorem}\label{thm:main2} If $u$ is a generalised Rudin--Shapiro sequence, then for any $r\in\Ns$, and any $(i,j)\in G^2$, 
$$\lim_{N\to \infty} {1\over N} \; \card\Big \{n \in \lb 0, N-1\rb : (u_n,u_{n+r})=(i,j) \Big \} =  {1\over |G|^2}.$$
\end{theorem}

\begin{remark}\label{rmk:main2}
Tahay obtained several results on the mean value of the discrete correlation coefficients along the integers. The \emph{discrete correlation coefficient} equals 1 if two symbols are identical, and 0 otherwise~\cite[Definition 1]{Ta20}. Theorem~\ref{thm:main2} gives a local result that is uniform in the values of the two symbols.
\end{remark}

\section{Discrete correlations of order 2 of generalised Rudin--Shapiro sequences}\label{Sec:RS2}

The aim of this section is to prove Theorem~\ref{thm:main1} and Theorem~\ref{thm:main2}. Namely, we prove that generalised Rudin--Shapiro sequences have the same discrete correlations of order 2 as i.i.d. sequences of symbols, and give a tight estimate of the speed of convergence. 

\subsection{Frequencies of letters in generalised Rudin--Shapiro sequences}

In this section, we present some first general results on generalised Rudin--Shapiro sequences, that we will need afterwards.

\begin{lemma} A generalised Rudin--Shapiro sequence is a primitive morphic sequence.
\end{lemma}

\begin{proof} As in the proof of Prop.~\ref{prop:aut}, let $Q=G\times \Sk$, and let $M$ be the matrix indexed by $Q$ and with values in $\{0,1\}$, defined by $M((g,i),(g',i'))=1$ if and only if there exists $j\in \Sk$ such that $(g',i')=(g+f(j,i),j)$. This matrix thus describes the allowed transitions in the DFAO given in the proof of Prop.~\ref{prop:aut}, or equivalently, the incidence matrix of the $k$-uniform morphism defined in Remark~\ref{rem:morphic}. We prove that all the entries of $M^{2|G|+3}$ are positive (the bound might be not optimal). Let $s_1=(g_1,i_1)$ and $s_2=(g_2,i_2)$ be two elements of $Q$. By the difference condition, there exists at least one $h\in G$ such that $f(1,h)-f(0,h)=g_2-g_1-f(0,1)-f(0,i_1)-f(i_2,0).$ From the state $i_1$, let us read in the DFAO the sequence $(0,h,0,h,0,h,\ldots,0,h,0,h,1,0,i_2)$, made of $|G|$ times the pattern $(0,h)$, followed by the pattern $(1,0,i_2)$. Then, the new state will be $s_2$, since
\begin{align*}
&f(0,i_1)+f(h,0)+f(0,h)+f(h,0)+\cdots+f(0,h)+f(h,0)+f(1,h)+f(0,1)+f(i_2,0)\\
&=|G|f(h,0)+(|G|-1)f(0,h)+f(1,h)+f(0,1)+f(0,i_1)+f(i_2,0)\\
&=f(1,h)-f(0,h)+f(0,1)+f(0,s_1)+f(s_2,0)=g_2-g_1.
\end{align*}
The conclusion follows.\end{proof}

\begin{proposition}\label{prop:freq} If $u$ is a generalised Rudin--Shapiro sequence, then any pattern has a frequency in the sequence $u$. Furthermore, the frequency of each element of $G$ (corresponding to patterns of length~1) is equal to ${1/ |G|}$.
\end{proposition}

\begin{proof} The existence of the frequencies for all patterns follows from the fact that the sequence $\phi^{\omega}(q_0)\in Q^{\N}$ is a primitive morphic sequence, where $\phi$ is the morphism given in Remark~\ref{rem:morphic}. Furthermore, each element of $Q$ has exactly $k$ preimages, since to state $s=(g,j)\in Q$, one can arrive from the state $(g-f(j,i),i)$, for any $i\in G$ (by reading $j$). So, all the elements of $Q$ have the same frequency in $\phi^{\omega}(q_0)$, and consequently, each element of $G$ has the same frequency in the image of $\phi^{\omega}(q_0)$ by $\tau$.
\end{proof}

\subsection{Fibre of an integer}

We now introduce the notion of \emph{fibre of an integer}, that will be useful in our context to study correlations of order $2$ of generalised Rudin--Shapiro sequences.

Let $r\in\N\setminus\{0\}$ be a fixed integer. For $n\in\N$, let us introduce the representations of $n$ and $n+r$ in base $k$ as follows
\begin{align*}
[n]_k&=x,\\
[n+r]_k&=y.
\end{align*}
We define the integer
$$c_n=\min\{i\in\N : \forall j>i, x_j=y_j\}.$$
Note that  $c_n$ depends on $r$, but that for the sake of shortness, we do not mention this dependence in the notation. The integer $c_n$ measures how far the carry propagates when adding $r$ to $n$. By definition, $x_{c_n}\not=y_{c_n}$ and $\forall j>c_n,$ $x_j=y_j$. We illustrate the definition of $c_n$ below.
\begin{equation}\label{eq:n}
\begin{array}{rccccccc}
[n]_k=&x_0&x_1&\cdots& {\color{blue}x_{c_n}}&x_{c_{n}+1}&x_{c_{n}+2}&\cdots \\
{[n+r]_k}=&y_0&y_1&\cdots& {\color{blue}y_{c_n}}&x_{c_{n}+1}&x_{c_{n}+2}&\cdots
\end{array}
\end{equation}

We define the \emph{fibre} of $n$ as the set
\begin{align*}
\F(n)&=\{m\in\N : x'=[m]_k \mbox{ satisfies } \forall i\in\N\setminus\{c_n+1\},\; x'_i=x_i\}\\
&=\{n+(\alpha-x_{c_n+1}) \, k^{c_n+1} : \alpha\in\Sk\}.
\end{align*}
We have thus
$$
\begin{array}{rccccccc}
\F(n)=\{ \; x_0&x_1&\cdots& x_{c_n}&{\small 0}&x_{c_{n}+2}&x_{c_{n}+3}&\cdots, \\
x_0&x_1&\cdots& x_{c_n}&{\small 1}&x_{c_{n}+2}&x_{c_{n}+3}&\cdots,\\
x_0&x_1&\cdots& x_{c_n}&{\small 2}&x_{c_{n}+2}&x_{c_{n}+3}&\cdots,\\
&&&&\vdots&&&\\
x_0&x_1&\cdots&x_{c_n}&{\small k-1}&x_{c_{n}+2}&x_{c_{n}+3}&\cdots\}.\\
\end{array}
$$
Note that if $m\in\F(n)$, then $c_m=c_n$, so that  
$$m\in \F(n) \iff n\in\F(m).$$ 
Furthermore, let $m\in\F(n)$, and let $x'=[m]_k$, $y'=[m+r]_k$. Then, we have 
$$y'_{c_n+1}=x'_{c_n+1}, \quad \mbox{ and } \quad \forall i\in \N\setminus\{c_n+1\}, \quad y'_i=y_i,$$
as represented below:
\begin{equation}\label{eq:m}
\begin{array}{rccccccc}
[m]_k=\, x'=&x_0&x_1&\cdots& x_{c_n}&x'_{c_{n}+1}&x_{c_{n}+2}&\cdots \\
{[m+r]_k}=\, y'=&y_0&y_1&\cdots& y_{c_n}&x'_{c_{n}+1}&x_{c_{n}+2}&\cdots
\end{array}
\end{equation}

Let $u$ be a block-additive sequence in base $k$ of weight $f$, and recall the notation $\pi=k/|G|$. For $n\in\N$, we also introduce the notation $\Dr(n)=u_{n+r}-u_n.$

\begin{proposition}\label{prop:fibre1} If $u$ is a generalised Rudin--Shapiro sequence, then for any $n\in\N$,  
$$\forall g\in G, \quad \card \{m\in\F(n) : \Dr(m) = g \} = \pi.$$
\end{proposition}

\begin{proof} By definition of a block-additive sequence, with the notations of~\eqref{eq:n}, we have
\begin{align*}
\Dr(n)&=\sum_{i\in\N} f(y_i,y_{i+1}) - \sum_{i\in\N} f(x_i,x_{i+1})\\  
&=\sum_{i=0}^{c_n} \Big ( f(y_i,y_{i+1}) - f(x_i,x_{i+1})\Big ).
\end{align*}
If $m\in\F(n)$, with the notations of \eqref{eq:m}, we have $$\Dr(m)=\sum_{i=0}^{c_n} (f(y'_i,y'_{i+1})-f(x'_i,x'_{i+1})),$$ so that
\begin{align*}
\Dr(m)-\Dr(n)&=\Big ( f(y'_{c_n},y'_{c_n+1}) - f(x'_{c_n},x'_{c_n+1}) \Big )
- \Big ( f(y_{c_n},y_{c_n+1}) - f(x_{c_n},x_{c_n+1}) \Big )\\
&=\Big ( f(y_{c_n},x'_{c_n+1}) - f(x_{c_n},x'_{c_n+1}) \Big )
- \Big ( f(y_{c_n},x_{c_n+1}) - f(x_{c_n},x_{c_n+1}) \Big )
\end{align*}
It follows that for all $g\in G$, 
$$\card \{m\in\F(n) : \Dr(m)-\Dr(n) =g \} = 
\card \Big \{ \alpha\in \Sk : f(y_{c_n},\alpha) - f(x_{c_n},\alpha) - A_n = g \Big\},$$
with $A_n=f(y_{c_n},x_{c_n+1}) - f(x_{c_n},x_{c_n+1}).$

Consequently, if $u$ is a generalised Rudin--Shapiro sequence, then for any $n\in\N$, we have
$$\forall g\in G, \quad \card \{m\in\F(n) : \Dr(m)-\Dr(n) = g \} = \pi,$$ 
and Prop.~\ref{prop:fibre1} follows.
\end{proof}

\subsection{Proof of Theorem~\ref{thm:main1}}

Using the notion of fibre developed above, we obtain the following proposition, from which Theorem~\ref{thm:main1} directly follows, since $\sum_{g\in G} \card\Big \{n \in \lb 0, N-1\rb : \Dr(n) = g \Big \}=N$. 

\begin{proposition}\label{prop:fibresRS} If $u$ is a generalised Rudin--Shapiro sequence, then for any $g\in G$,
\begin{align*}
\card\Big \{n \in \lb 0, N-1\rb : \Dr(n) = g \Big \} & \geq {\pi N\over k} - \pi \, r\, k - \pi \, r \, \sigma_k(N)\\
& \geq {N\over |G|} - \pi \, r\, k (1+\log_k(N)).\\
\end{align*}
\end{proposition}

\begin{proof} Let $N\in\Ns$, and let $a=[N]_k$. We determine the conditions under which an integer $n\in\lb 0, N-1\rb$ satisfies $\F(n)\subset\lb 0, N-1\rb$. Recall the notation $\ell_N=\min\{i\in \N : \forall j>i, a_i=0\}$. We can thus write
$$[N]_k= a_0 \; a_1 \; \cdots \; a_{\ell_N-1} \; a_{\ell_N} \; 0 \; 0 \; \cdots $$

\begin{itemize}
\item If $n=a'_{\ell_N}\, k^{\ell_N}+\alpha\, k^{\ell_N-1}+\gamma$, for some $\alpha\leq k-1$, $a'_{\ell_N}<a_{\ell_N}$, and $\gamma<k^{\ell_N-1}-r$, then $c_n\leq \ell_N-2$, so that $\F(n)\subset\lb 0, N-1\rb$.
\begin{align*}
[n]_k&= \underbrace{x_0 \; x_1 \; \cdots \; x_{\ell_N-2}}_{\gamma<k^{\ell_{N}-1}-r} \; \alpha \; \underbrace{a'_{\ell_N}}_{<a_{\ell_N}} \; 0 \; 0 \; \cdots \\
[n+r]_k&= x'_0 \; x'_1 \; \cdots \; x'_{\ell_N-2} \;\; \alpha \;\; a'_{\ell_N} \;\; 0 \; 0 \; \cdots \\
\end{align*}

\item If $n=a_{\ell_N}\, k^{\ell_N}+a'_{\ell_N-1}\, k^{\ell_N-1}+\alpha\, k^{\ell_N-2}+\gamma$, for some $\alpha\leq k-1$, $a'_{\ell_N-1}<a_{\ell_N-1}$, and $\gamma<k^{\ell_N-2}-r$, then $c_n\leq \ell_N-3$, so that $\F(n)\subset\lb 0, N-1\rb$.
\begin{align*}
[n]_k&= \underbrace{x_0 \; x_1 \; \cdots \; x_{\ell_N-3}}_{\gamma<k^{\ell_N-2}-r} \; \alpha \; \underbrace{a'_{\ell_N-1}}_{<a_{\ell_N-1}} \; a_{\ell_N} \; 0 \; 0 \; \cdots \\
[n+r]_k&= x'_0 \; x'_1 \; \cdots \; x'_{\ell_N-3} \;\; \alpha \;\; a'_{\ell_N-1} \;\; a_{\ell_N} \; 0 \; 0 \; \cdots \\
\end{align*}

\item If $n=a_{\ell_N}\, k^{\ell_N}+a_{\ell_N-1}\, k^{\ell_N-1}+a'_{\ell_N-2}\, k^{\ell_N-2}+\alpha\, k^{\ell_N-3}+\gamma$, for some $\alpha\leq k-1$, $a'_{\ell_N-2}<a_{\ell_N-2}$, and $\gamma<k^{\ell_{N-3}}-r$, then $c_n\leq \ell_N-4$, so that $\F(n)\subset\lb 0, N\rb$.
\begin{align*}
[n]_k&= \underbrace{x_0 \; x_1 \; \cdots \; x_{\ell_N-4}}_{\gamma<k^{\ell_N-3}-r} \; \alpha \; \underbrace{a'_{\ell_N-2}}_{<a_{\ell_N-2}} \; a_{\ell_N-1} \; a_{\ell_N} \; 0 \; 0 \; \cdots \\
[n+r]_k&= x'_0 \; x'_1 \; \cdots \; x'_{\ell_N-4} \;\; \alpha \;\; a'_{\ell_N-2}  \;\; a_{\ell_N-1} \; a_{\ell_N} \; 0 \; 0 \; \cdots \\
\end{align*}

\item And so on, the last condition that will be of interest for us being that if $n=a_{\ell_N}\, k^{\ell_N}+a_{\ell_N-1}\, k^{\ell_N-1}+\ldots+a_{\ell_r+3}\, k^{\ell_r+3} + a'_{\ell_r+2} \, k^{\ell_r+2} + \alpha \, k^{\ell_r+1} + \gamma$, for some $\alpha\leq k-1$, $a'_{\ell_r+2}<a_{\ell_r+2}$, and $\gamma<k^{\ell_r+1}-r$, then $c_n\leq \ell_r$, so that $\F(n)\subset\lb 0, N-1\rb$.
\end{itemize}

The number of different integers $n\in\lb 0, N-1\rb$ satisfying $\F(n)\subset\lb 0, N-1\rb$ that we have exhibited above is equal to
\begin{align*}
\; & \; a_{\ell_N}k (k^{\ell_N-1}-r) + a_{\ell_{N-1}}k (k^{\ell_N-2}-r)+a_{\ell_{N-2}}k (k^{\ell_N-3}-r)+ \ldots + a_{\ell_{r}+2}k (k^{\ell_r+1}-r)\\
= \; & \; N-(a_{\ell_r+1}k^{\ell_r+1}+a_{\ell_r}k^{\ell_r}+\ldots+a_1k+a_0)-r\, k\, (a_{\ell_N}+a_{\ell_N-1}+a_{\ell_N-2}+\ldots+a_{\ell_r+2})\\
> \; & \; N-r \, k^2 - r \, k \, \sigma_k(N).
\end{align*}
For the last inequality, observe that $a_{\ell_r+1}k^{\ell_r+1}+a_{\ell_r}k^{\ell_r}+\ldots+a_1k+a_0<k^{\ell_r+2}\leq r\, k^2$. 
Prop.~\ref{prop:fibresRS} then directly follows from Prop.~\ref{prop:fibre1}. \end{proof}

\subsection{Correlation matrix}

In order to prove Theorem~\ref{thm:main2}, we first introduce the notion of correlation matrix, and formulate the previous results using this terminology.

Let $u\in G^{\N}$ be a fixed sequence. For $r\in\Ns, (i,j) \in G^2$ and $n\in\N$, we define
$$\dr_{i,j}(n)=
\begin{cases} 
1 & \mbox{ if } (u_n,u_{n+r})=(i,j),\\
0 & \mbox{ otherwise.}
\end{cases}$$
and 
$$C^r_{i,j}(N)={1\over N} \sum_{n=0}^{N-1} \dr_{i,j}(n).$$

As a consequence of Prop.~\ref{prop:freq}, if $u$ is a generalised Rudin--Shapiro sequence, then for any $r\in\Ns$ and $(i,j)\in G^2$, the sequence $C^r_{i,j}(N)$ converges when $N$ goes to infinity, so that we can also introduce
$$C^r_{i,j}=\lim_{N\to\infty}C^r_{i,j}(N).$$
Furthermore, again by Prop.~\ref{prop:freq}, for any $i\in G$, the asymptotic frequency of the symbol $i$ is 
$$\sum_{j\in G} C^r_{i,j} = {1\over |G|}.$$

As a consequence of Prop.~\ref{prop:fibresRS}, we obtain the following results.

\begin{corollary}\label{cor:corrRS} 
If $u$ is a generalised Rudin--Shapiro sequence, then for any $(i,j)\in G^2$,
$$\sum_{\ell\in G} C^r_{i-\ell,j-\ell}(N) \geq {1\over |G|} - \pi \, r\, k {1+\log_k(N) \over N}.$$
\end{corollary}

\begin{corollary}\label{cor:corrRSdiff} If $u$ is a generalised Rudin--Shapiro sequence, then for any $(i,j)\in G^2$,
$$\sum_{\ell\in G} C^r_{i-\ell,j-\ell} = {1\over |G|}.$$
\end{corollary}

\begin{proof} 
It is a consequence from Cor.~\ref{cor:corrRS} and the observation that $\sum_{(i,j)\in G^2} C^r_{i,j}=1$.
\end{proof}

Note that this result refines the estimates of Tahay concerning the discrete correlation coefficient (cf. Remark~\ref{rmk:main2}) that detects whether two symbols differ or not. In our language, he proved that
$$\sum_{i\in G} C^r_{i,i} = {1\over |G|}.$$

\subsection{Proof of Theorem~\ref{thm:main2}}

With the notations above, Theorem~\ref{thm:main2} is equivalent to next proposition, that we now prove. Note that this result is stronger than Corollary~\ref{cor:corrRSdiff} as it gives the values of the individual terms in the sum.

\begin{proposition}\label{prop:cormat} If $u$ is a generalised Rudin--Shapiro sequence, then for any $(i,j)\in G^2$, 
$$C^r_{i,j}= {1\over |G|^2}.$$
\end{proposition}

\begin{proof} 
Let us fix some $\alpha\in\Sk$ and consider the integers $n\in \lb 0, k^{2N+1}-1\rb$ that are such that the base-$k$ decomposition $x=[n]_k$ of $n$ satisfies $x_{N+1}=\alpha$. In other words, $n=m_1\, k^{N+1} + \alpha \, k^{N} + m_2$, for some integers $m_1,m_2\in \lb 0, k^{N}-1\rb$. 
Assuming furthermore that $m_2<k^N-r$, we will have $c_n<N$, so that
$$(u_n,u_{n+r})=(u_{km_1+\alpha},u_{km_1+\alpha})+(u_{\alpha k^{N} + m_2},u_{\alpha k^{N} + m_2+r}),$$
by definition of a block-additive sequence. 

The proof will be based on the following idea: when taking independently at random some integers $m_1,m_2$ uniformly distributed in $\lb 0, k^{N}-1\rb$, the distribution of $u_{km_1+\alpha}$ converges to the uniform distribution on $G$ when $N$ goes to infinity, while for the second term $(u_{\alpha k^{N} + m_2},u_{\alpha k^{N} + m_2+r})$, the distribution is asymptotically given by the values $C_{i,j}$ of the correlation matrix. 
Now, we have $(u_n,u_{n+r})=(i,j)$ if $u_{km_1+\alpha}=\ell$ for some $\ell$ and $(u_{\alpha k^{N} + m_2},u_{\alpha k^{N} + m_2+r})=(i-\ell,j-\ell)$. Using the independence of $m_1$ and $m_2$, we thus obtain
$$C^r_{i,j}=\sum_{\alpha\in \Sk}  {1\over k} \sum_{\ell \in G} {1\over |G|} C_{i-\ell,j-\ell}=\sum_{\alpha\in \Sk}  {1\over k}{1\over |G|} {1\over |G|}={1\over |G|^2},$$
since we already know by Corollary~\ref{cor:corrRSdiff} that for any $(i,j)\in G^2, \sum_{\ell\in G} C_{i-\ell,j-\ell}={1\over |G|}$.

More formally, let us introduce the following notations, for any $i,j,\ell\in G$,
\begin{align*}
A^{\alpha}_{\ell}(N)&=\card\{m\in\lb 0, k^N-1\rb : u_{km+\alpha}=\ell\}\\
B^{r,\alpha}_{i,j}(N)&=\card\{m\in\lb 0, k^N-r-1\rb : \delta_{i,j}^r(\alpha\, k^{N}+m)=1\}.
\end{align*}
We claim that for any $\alpha\in \Sk$,
$$\lim_{N\to\infty} {A^{\alpha}_{\ell}(N)\over k^N}={1\over |G|}, \quad \mbox{ and } \quad \lim_{N\to\infty} \sum_{\ell\in G} {B^{r,\alpha}_{i-\ell,j-\ell}(N)\over k^N}={1\over |G|}.$$

For the first limit, we use the same tools as for Prop.~\ref{prop:freq}. Let $\phi$ be the primitive morphism given in Remark~\ref{rem:morphic}, so that the sequence $u$ is the image of $\phi^{\omega}(q_0)$ by $\tau$. One can see that the sequence $(u_{kn+\alpha})_{n\in\N}$ is the image of $\phi^{\omega}(q_0)$ by the function $\tau':Q\to G$ defined by $\tau'(g,i)=g+f(i,\alpha)$. As we have already seen in the proof of Prop.~\ref{prop:freq}, all the elements of $Q$ have the same frequency in $\phi^{\omega}(q_0)$. Consequently, each element of $G$ has the same frequency in the image of $\phi^{\omega}(q_0)$ by $\tau'$. Indeed, for any $g'\in G$ and $i\in\Sk$, there exists exactly one $g\in G$ such that $\tau'(g,i)=g'$, so that the cardinal of $\tau'^{-1}(\{g'\})$ does not depend on the choice of $g'$.

The second limit is a small variation of Cor.~\ref{cor:corrRSdiff}, and it can be proven exactly in the same way, using the same steps as in Prop.~\ref{prop:fibresRS}. 
Furthermore, for any $(i,j)\in G^2$, we have
$$\sum_{n=0}^{k^{2N+1}-1} \delta_{i,j}^r(n) \geq \sum_{\alpha\in \Sk} \sum_{\ell\in G} A^N_{\ell}(\alpha) \; B^N_{i-\ell,j-\ell}(\alpha).$$
It follows that
$$C^r_{i,j}(k^{2N+1}) \geq {1\over k} \sum_{\alpha\in \Sk} \sum_{\ell\in G} {A^N_{\ell}(\alpha)\over k^N} \; {B^N_{i-\ell,j-\ell}(\alpha)\over k^N}.$$
When $N$ goes to infinity, we know that the limit of the left term exists and is equal to $C^r_{i,j}$. We thus obtain
$$C^r_{i,j} \geq {1\over |G|^2}.$$
Since $\sum_{(i,j)\in G^2}C^r_{i,j}=1$, this ends the proof.
\end{proof}

\section{Higher dimensional generalised Rudin--Shapiro sequences}

We propose the following natural extension of Def.~\ref{def:blockadd} and~\ref{def:grs}  in dimension $d$. For greater readability, we represent the elements of $\Sk^d$ as column vectors.

\begin{definition} 
Let $(G,+)$ be a finite abelian group, and let $k\in\Ns$. We say that the sequence $u=(u_{n_1,\ldots,n_d})_{(n_1,\ldots,n_d)\in\N^d}\in G^{\N^d}$ is a \emph{$d$-dimensional block-additive sequence in base $k$} if there exists a map $f:\Sk^d\times \Sk^d\to G$ satisfying $f\left(\myvec{0\\ \vdots\\0},\myvec{0\\ \vdots\\0}\right)=0$, such that for any integer $n\in\N$, we have
$$u_{n_1,\ldots,n_d}=\sum_{i\in\N} f\left(\myvec{x^1_i\\\vdots\\x^d_i},\myvec{x^1_{i+1}\\\vdots\\x^d_{i+1}}\right)=\sum_{i\in\N} f(x_i,x_{i+1}),$$
where $x=(x_i)_{i\in\N}=\myvec{x^1\\ \vdots \\x^d}=\myvec{(x_i^1)_{i\in\N}\\ \vdots \\(x_i^d)_{i\in\N}}=\myvec{[n_1]_k\\ \vdots \\ [n_d]_k}$.

We say furthermore that the sequence $u$ is a \emph{generalised $d$-dimensional Rudin--Shapiro sequence} if the function $f$ satisfies
$$\forall (i,j)\in \Sk^d\times \Sk^d \mbox{ with }\; i\not= j, \quad \forall g\in G, \quad \card \Big \{h\in \Sk^d : f(i,h)-f(j,h) = g \Big \} ={k\over |G|}.$$
Equivalently, this amounts to saying that the matrix $(f(i,j))_{(i,j)\in \Sk^d\times\Sk^d}$ is a difference matrix. 
\end{definition}

As in the one-dimensional case, a $d$-dimensional sequence that is block-additive in base $k$ is a $k$-automatic sequence.

Let $r\in\N^d\setminus\{(0,\ldots,0)\}$. For $n=(n_1,\ldots,n_d)\in\N^d$, we introduce the representations of $n$ and $n+r$ in base $k$ as follows
\begin{align*}
[n]_k=x=\myvec{x^1 \\ \vdots \\ x^d}, \qquad
[n+r]_k=y=\myvec{y^1 \\ \vdots \\ y^d},
\end{align*}
and we define the integer
$$c_n=\min\{i\in\N : \;\forall j>i, \, x_j=y_j\},$$
which measures how far the carry propagates when adding $r$ to $n$.

We define again the \emph{fibre} of $n$ as the set
$$\F(n)=\{m\in\N : x'=[m]_k \mbox{ satisfies } \forall i\in\N\setminus\{c_n+1\},\; x'_i=x_i\},$$
and use the notation $\Dr(n)=u_{n+r}-u_n.$

Since the $d$-dimensional sequence has $d$ components that are all 1-dimensional and independent, the previous arguments can be repeated verbatim.

\begin{proposition} If $u$ is a generalised $d$-dimensional Rudin--Shapiro sequence, then for any $n\in\N$,  
$$\forall g\in G, \quad \card \{m\in\F(n) : \Dr(m) = g \} = \pi.$$
\end{proposition}

We also extend the notations $\dr$ and $C^r$ to $d$-dimensional sequences. Precisely, for $N=(N_1,\ldots,N_d)$, we define
$$C^r_{i,j}(N)={1\over {N_1\cdots N_d}} \sum_{\{n\in\N^d \, : \, n<N\}} \dr_{i,j}(n),$$
where we write $n<N$ for $\forall i\in\{1,\ldots,d\}, n_i<N_i$. We also introduce
$$C^r_{i,j}=\lim_{N\to\infty}C^r_{i,j}(N).$$

Following the previous lines, one can show as in the one-dimensional case that if $u$ is a generalised $d$-dimensional Rudin--Shapiro sequence, then for any $(i,j)\in G^2$,
$$\sum_{\ell\in G} C^r_{i-\ell,j-\ell} = {1\over |G|},$$
which also allows to obtain the following extension of Prop.~\ref{prop:cormat}.

\begin{proposition} If $u$ is a generalised $d$-dimensional Rudin--Shapiro sequence, then for any $(i,j)\in G^2$, 
$$C^r_{i,j}= {1\over |G|^2}.$$
\end{proposition}

\begin{example} We present in Fig.~\ref{fig:2D} four different examples of generalised Rudin--Shapiro sequence, for $d=2$, $k=2$, $G=\Z_2$. For each example, the values of the function $f:\Sigma_2^2\to\Z_2$ is given by a matrix, with the elements of $\Sigma_2^2$ sorted in the lexicographic order. On the first line of the matrix, one can thus read successively 
$$f\left(\myvec{0\\0},\myvec{0\\0}\right),f\left(\myvec{0\\0},\myvec{0\\1}\right),f\left(\myvec{0\\0},\myvec{1\\0}\right),f\left(\myvec{0\\0},\myvec{1\\1}\right),$$ 
and then on the second line $$f\left(\myvec{0\\1},\myvec{0\\0}\right),f\left(\myvec{0\\1},\myvec{0\\1}\right),\ldots$$ 
and so on. On the pictures, the cell $(n_1,n_2)\in\N^2$ is colored in blue if $u_{n_1,n_2}=1$ and in white if $u_{n_1,n_2}=0$. The corner corresponding to the value $u_{0,0}$ is the bottom-left corner.

Let us present more in detail the first example. For $i,j\in\Sigma_2^2$, the weight function satisfies $f(i,j)=0$ if $i=j$, and $f(i,j)=1$ otherwise. As an example, we compute below $u_{436,48}$. 
$$\begin{array}{cccccccccccccl}
[436]_2 & = & 0 & 0 & 1 & 0 & 1 & 1 & 0 & 1 & 1 & 0 & 0 &\cdots \\
{[48]_2} &= & 0 & 0 & 0 & 0 & 1 & 1 & 0 & 0 & 0 & 0 & 0 & \cdots \\ 
\hline
u_{436,48} & \equiv && \hspace{-0.2cm} 0 \; +& \hspace{-0.2cm} 1 \; +& \hspace{-0.2cm} 1 \; +& \hspace{-0.2cm} 1 \; +& \hspace{-0.2cm} 0\; +& \hspace{-0.2cm} 1\; +& \hspace{-0.2cm} 1\; +& \hspace{-0.2cm} 0\; +& \hspace{-0.2cm} 1\; +& \hspace{-0.2cm} 0 \; +& \cdots \; \equiv 0 \pmod{2}
\end{array}$$
The following table gives the first values of $u_{n_1,n_2}$, for $(n_1,n_2)\in \lb 0, 2^3-1\rb^2$.
$$
\begin{array}{c|ccccccccc}
7 &1 &0 &1 &0 &0 &1 &0 &1\\
6 &0 &0 &1 &1 &1 &1 &0 &0\\
5 &1 &1 &1 &1 &1 &1 &1 &1\\
4 &0 &1 &1 &0 &0 &1 &1 &0\\
3 &1 &0 &0 &1 &0 &1 &1 &0\\
2 & 0 &0 &0 &0 &1 &1 &1 &1\\
1&1 &1 &0 &0 &1 &1 &0 &0\\
0&0 &1 &0 &1 &0 &1 &0 &1\\ \hline
{\tiny \;^{n_2}\hspace{-0.1cm}\diagup\hspace{-0.2cm}\;_{n_1}}\hspace{-0.2cm}\;&0&1&2&3&4&5&6&7\\
\end{array}
$$
These values are also contained in the bottom-left $8\times 8$-squares of the two pictures on the first line of Fig.~\ref{fig:2D}.

Concerning the second example, it can be seen that the weight functions satisfies 
$$f\left(\myvec{i_1\\i_2},\myvec{j_1\\j_2}\right)\equiv i_1j_1+i_2j_2 \pmod{2}.$$
As a consequence, the sequence obtained can also be computed by $u_{m,n}=v_m+v_n$, where $v$ is the classical one-dimensional Rudin--Shapiro sequence.

\begin{figure}
\newcolumntype{C}{>{\centering\arraybackslash}m{2.5cm}}
\newcolumntype{D}{>{\centering\arraybackslash}m{6cm}}
\begin{tabular}{|C|D|D|}
\hline
Matrix & Terms in $\lb 0, 2^7-1\rb^2$ & Terms in $\lb 0, 2^{10}-1\rb^2$ \\
\hline
$\begin{pmatrix}
0&1&1&1\\
1&0&1&1\\
1&1&0&1\\
1&1&1&0
\end{pmatrix}$
&\reflectbox{\rotatebox[origin=c]{180}{\includegraphics[scale=0.45]{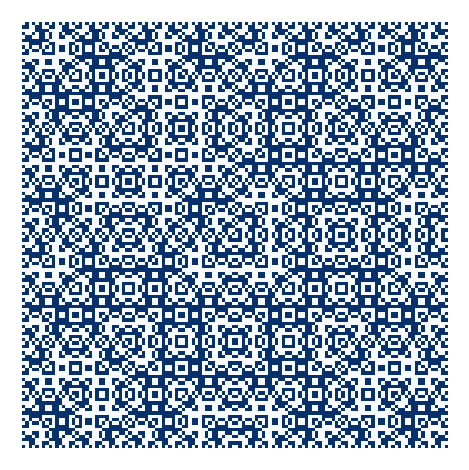}}}
&\reflectbox{\rotatebox[origin=c]{180}{\includegraphics[scale=0.45]{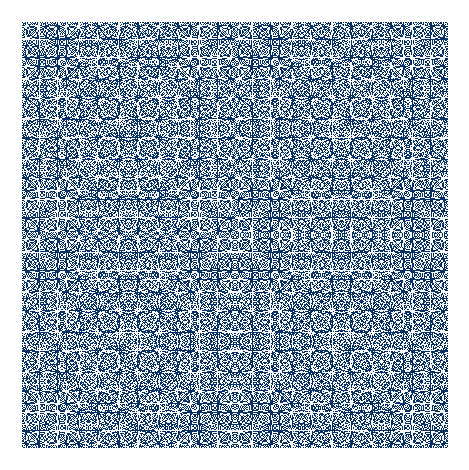}}}\\
\hline
$\begin{pmatrix}
0&0&0&0\\
0&1&0&1\\
0&0&1&1\\
0&1&1&0
\end{pmatrix}$
&\reflectbox{\rotatebox[origin=c]{180}{\includegraphics[scale=0.45]{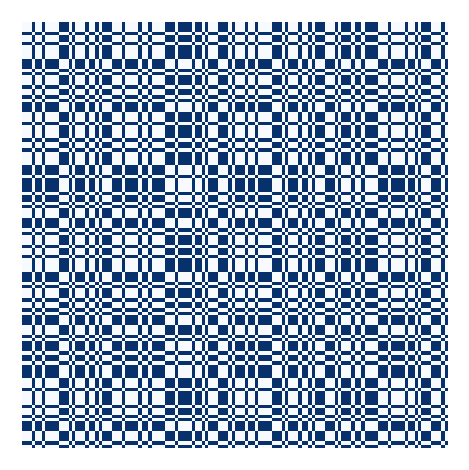}}}
&\reflectbox{\rotatebox[origin=c]{180}{\includegraphics[scale=0.45]{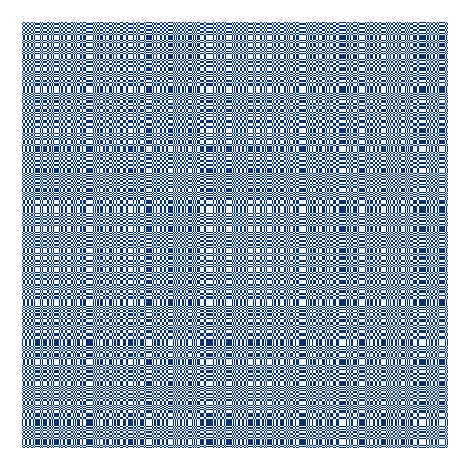}}}\\
\hline
$\begin{pmatrix}
0&0&0&0\\
0&0&1&1\\
0&1&0&1\\
0&1&1&0
\end{pmatrix}$
&\reflectbox{\rotatebox[origin=c]{180}{\includegraphics[scale=0.45]{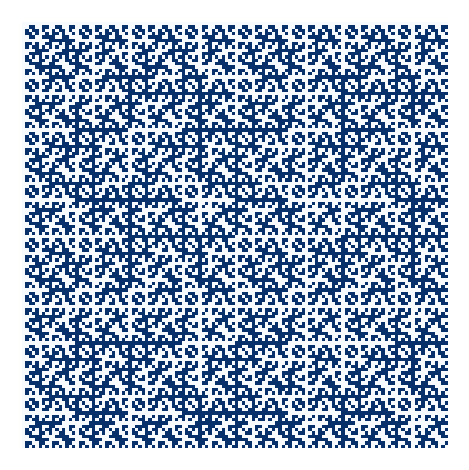}}}
&\reflectbox{\rotatebox[origin=c]{180}{\includegraphics[scale=0.45]{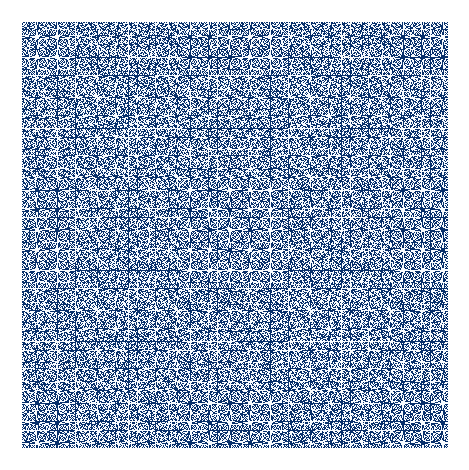}}}\\
\hline
$\begin{pmatrix}
0&0&1&1\\
0&1&0&1\\
0&0&0&0\\
0&1&1&0
\end{pmatrix}$
&\reflectbox{\rotatebox[origin=c]{180}{\includegraphics[scale=0.45]{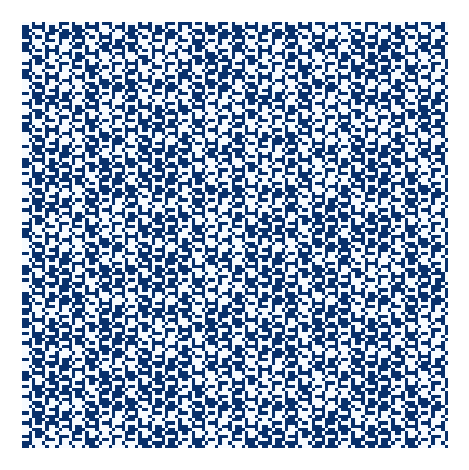}}}
&\reflectbox{\rotatebox[origin=c]{180}{\includegraphics[scale=0.45]{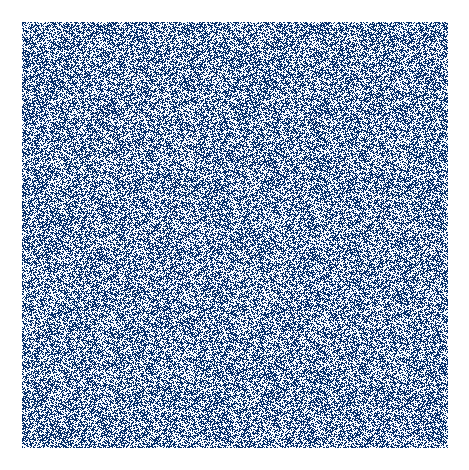}}}\\
\hline
\end{tabular}
\caption{Examples of generalised $2$-dimensional Rudin--Shapiro sequence in base $2$}
\end{figure}\label{fig:2D}

\end{example}

\section{Extensions and open questions}

\subsection{Block-additive sequences of rank larger than 2}

Until now, we have only considered block-additive of rank 2. More generally, we can consider the notion of block-additive function of rank $L$, for an integer $L\in\Ns$, in the sense of Cateland~\cite{Ca92}.

\begin{definition} 
Let $(G,+)$ be a finite abelian group, let $k\in\Ns$, and let $f:\Sk^L\to G$ be a function satisfying $f(0,0,\ldots,0)=0$. We say that the sequence $u=(u_n)_{n\in\N}\in G^{\N}$ is a \emph{block-additive sequence (of rank $L$) in base $k$} of \emph{weight function} $f$ if for any integer $n\in\N$, we have
$$u_n=\sum_{i\in\N} f(x_i,x_{i+1},\ldots,x_{i+L-1}),$$
where $[n]_k=x$.
\end{definition}

Let $(G,+)$ be a finite abelian group, and let $k\in\Ns$. We say that the function $d:\Sk^L\to G$ satisfies the \emph{difference condition (of rank $L$)} if: 
\begin{align*}
&\forall (i,j)\in \Sk\times\Sk \mbox{ with }\; i\not= j, \quad \forall (x_2,\ldots,x_{L-1})\in \Sk^{L-2},\\
&\forall g\in G, \quad \card \{h\in \Sk : d(i,x_2,\ldots,x_{L-1},h)-d(j,x_2,\ldots,x_{L-1},h) = g \} ={k\over |G|}.
\end{align*}

The difference condition is a sufficient condition for obtaining the same results as in Section~\ref{Sec:RS2}. 

\begin{example} Let us set $k=2$, $G=\Z_2$, and let $f:\Sk^3\to G$ be defined by $f(x,y,z)=\begin{cases} 0 \mbox{ if } x=y=z\\ 1 \mbox{ otherwise.} \end{cases}$.
This function satisfies the difference condition. Consequently, the block-additive sequence $u=(u_n)_{n\in\N}$ of weight function $f$, which is such that $u_n$ counts (modulo $2$) the number of blocks different from $000$ and $111$ in the binary representation of $n$, has the same correlations of order $2$ as a binary sequence chosen uniformly at random. 
\end{example}

\begin{open} How can we generate functions satisfying the difference condition of rank $L$? Could there be a weaker condition on the weight function for which the block-additive sequences obtained would have the same correlations of order $2$?
\end{open}

\subsection{Can an automatic sequence look even more random?}

Another possible direction of research consists in trying to construct block-additive sequences for which not only the correlations of order $2$, but also correlations of higher order would be the same as for uniform random sequences. Precisely, for integers $0<r_1<\ldots<r_{\ell-1}$, and for a choice $(i_0,\ldots,i_{\ell-1})\in G^{\ell}$, we introduce 
$$\dr_{i_0,\ldots,i_{\ell-1}}(n)=
\begin{cases} 
1 & \mbox{ if } (u_n,u_{n+r_1},\ldots,u_{n+r_{\ell-1}})=(i_0,\ldots,i_{\ell-1}),\\
0 & \mbox{ otherwise,}
\end{cases}$$
and we look at the asymptotic behaviour of ${1\over N}\sum_{n=0}^{N-1} \dr_{(i_0,\ldots,i_{\ell-1})}(n)$, when $N$ goes to infinity. We say that a sequence has the same correlations of order $\ell$ as a uniform random sequence if for any choice of $0<r_1<\ldots<r_{\ell-1}$, and for any $(i_0,\ldots,i_{\ell-1})\in G^{\ell}$, 
$$\lim_{N\to\infty} {1\over N} \sum_{n=0}^{N-1} \dr_{i_0,\ldots,i_{\ell-1}}(n)={1\over |G|^{\ell}}.$$

\begin{open} For a given $\ell\geq 3$, is it possible to construct a block-additive sequence having the same correlations of order $\ell$ as a uniform random sequence?
\end{open}

Note that it is not possible to construct an automatic sequence such that \emph{for any $\ell\geq 1$}, the correlations of order $\ell$ would be the same as for a uniform random sequence. Indeed, this would in particular imply the sequence to be normal, while the complexity of an automatic sequence is at most linear.

\bibliographystyle{plainurl}
\bibliography{biblio.bib}

\begin{thebibliography}{10}

\bibitem{AS93}
Jean-Paul Allouche and Jeffrey Shallit.
\newblock Complexit\'{e} des suites de {R}udin-{S}hapiro
  g\'{e}n\'{e}ralis\'{e}es.
\newblock {\em J. Th\'{e}or. Nombres Bordeaux}, 5(2):283--302, 1993.

\bibitem{AlSh03}
Jean-Paul Allouche and Jeffrey Shallit.
\newblock {\em Automatic sequences}.
\newblock Cambridge University Press, Cambridge, 2003.
\newblock Theory, applications, generalizations.

\bibitem{Ca92}
Emmanuel Cateland.
\newblock {\em {Digital sequences and k-regular sequences}}.
\newblock Theses, {Universit{\'e} Sciences et Technologies - Bordeaux I}, June
  1992.

\bibitem{DrGrLi08}
Michael Drmota, Peter~J. Grabner, and Pierre Liardet.
\newblock Block additive functions on the {G}aussian integers.
\newblock {\em Acta Arith.}, 135(4):299--332, 2008.

\bibitem{GrShSt09}
Elyot Grant, Jeffrey Shallit, and Thomas Stoll.
\newblock Bounds for the discrete correlation of infinite sequences on {$k$}
  symbols and generalized {R}udin-{S}hapiro sequences.
\newblock {\em Acta Arith.}, 140(4):345--368, 2009.

\bibitem{HeSlSt99}
A.~S. Hedayat, N.~J.~A. Sloane, and John Stufken.
\newblock {\em Orthogonal arrays}.
\newblock Springer Series in Statistics. Springer-Verlag, New York, 1999.
\newblock Theory and applications, With a foreword by C. R. Rao.

\bibitem{La15}
P.~H.~J. Lampio.
\newblock {\em Classification of difference matrices and complex {Hadamard}
  matrices}.
\newblock PhD thesis, Aalto University, 2015.

\bibitem{MR15}
Christian Mauduit and Jo\"{e}l Rivat.
\newblock Prime numbers along {R}udin-{S}hapiro sequences.
\newblock {\em J. Eur. Math. Soc. (JEMS)}, 17(10):2595--2642, 2015.

\bibitem{Mu18}
Clemens M\"{u}llner.
\newblock The {R}udin-{S}hapiro sequence and similar sequences are normal along
  squares.
\newblock {\em Canad. J. Math.}, 70(5):1096--1129, 2018.

\bibitem{Qu87}
Martine Queff\'{e}lec.
\newblock Une nouvelle propri\'{e}t\'{e} des suites de {R}udin-{S}hapiro.
\newblock {\em Ann. Inst. Fourier (Grenoble)}, 37(2):115--138, 1987.

\bibitem{Ta20}
Pierre-Adrien Tahay.
\newblock {D}iscrete correlation of order 2 of generalized {R}udin-{S}hapiro
  sequences on alphabets of arbitrary size.
\newblock {\em Unif. Distrib. Theory}, 15(1):1--26, 2020.

\end{thebibliography}

\end{document}